\documentclass[12pt]{amsart}
\usepackage[utf8]{inputenc}
\usepackage{amsmath,amssymb,amsthm}
\usepackage{tikz,tikz-cd} 

\DeclareMathOperator{\Pic}{\mathrm{Pic}}

\newtheorem{conjecture}[equation]{Conjecture} 
 
\newtheorem{theorem}[equation]{Theorem}
\newtheorem{question}[equation]{Question}
\newtheorem{lemma}[equation]{Lemma}
\newtheorem{corollary}[equation]{Corollary}
\theoremstyle{definition}
\newtheorem{definition}{Definition}[section]
\newtheorem{claim}[equation]{Claim}
\theoremstyle{remark}
\newtheorem{remark}[equation]{Remark}

\begin{document}
\begin{abstract}In this paper, we use canonical bundle formulas to prove various generalizations of an old theorem of Kawamata on the semiampleness of nef and abundant log canonical divisors. In particular, we show that for klt pairs $(X,B)$ with $K_X+B$ effective, $L \in \Pic X$ nef, nefness and abundance of $K_X+B+L$ implies semiampleness. This essentially extends Kawamata's theorem to the setting of generalized abundance.

\end{abstract}

\title[Nef and abundant divisors]{Nef and abundant divisors, semiampleness and canonical bundle formula}
\author{Priyankur Chaudhuri}
\address{Department of Mathematics, University of Maryland, College Park, MD 20742}
\email{pchaudhu@umd.edu}

\maketitle
 
\section{Introduction}
This paper will be concerned with the following:
\begin{question} 
Let $(X, B)$ be a klt pair and $ L \in \Pic X$ nef such that $K_X+B+L$ is nef and abundant, then is $K_X+B+L$ semiample?
\end{question}
This is clearly related to the following conjecture of Lazić and Peternell \cite{LP}.

\begin{conjecture} (Generalized Abundance) Let $(X,B)$ be a klt pair with $K_X+B$ pseudoeffective and $L$ a nef Cartier divisor on $X$ such that $K_X+B +L$ is also nef. Then there exists a semiample $\mathbb{Q}$-divisor $M$ on $X$ such that $K_X+B+L \equiv M$.
\end{conjecture}
\let\thefootnote\relax\footnote{2020 Mathematics Subject Classification: 14E30}

Note that we are not assuming pseudoeffectivity of $(K_X+ B)$ in our question.
There are two reasons why one might expect this to be true. Firstly, in \cite{PC}, the author has shown that generalized abundance does continue to hold for many uniruled varieties. Secondly, if $L \in \Pic X$ is a large multiple of a strictly nef divisor, then under the above assumptions, one can show that $K_X+B+L$ is ample. So it is natural to ask for a version in which the pseudoeffectivity of $K_X+B$ is not required. The surface case of this question is very easily seen to be true. We now state our main results which settle certain special cases. Essentially, they show that Question 1 has a positive answer in the following two situations:
\begin{enumerate}

\item If $L$ is a birational pushforward of a semiample divisor (in fact, it does not need to be nef here).

\item If $K_X+B \geq 0$. 

\end{enumerate}

\begin{theorem} Let $(X,B+\mathbf{M})$ be a generalized klt pair, $H$ a nef $\mathbb{Q}$-Cartier divisor on $X$ such that $H-(K_X+B+\mathbf{M}_X)$ is nef and abundant. Assume further that $\mathbf{M}$ is b-semiample, $\kappa(aH-(K_X+B+\mathbf{M}_X)) \geq 0$ and $ \nu(aH-(K_X+B+\mathbf{M}_X)) =\nu(H-(K_X+B+\mathbf{M}_X))$ for some $a \in \mathbb{Q}_{>1}$. Then $H$ is semiample.
\end{theorem}

The case $\mathbf{M}=0$ of the above theorem is due to Kawamata \cite{Ka} and Fujino \cite{Fu}. Our approach of using canonical bundle formula is similar to that of Fujino's. 

\begin{theorem}
Let $(X, B)$ be a klt pair with $\kappa(K_X+B) \geq 0$, $H$ and $L$ nef $\mathbb{Q}$-Cartier divisors on $X$ such that $H-(K_X+B+L)$ is nef and abundant. Assume further that $\kappa(aH-(K_X+B+L)) \geq 0$ and $ \nu(aH-(K_X+B+L)) =\nu(H-(K_X+B+L))$ for some $a \in \mathbb{Q}_{>1}$. Then $H$ is semiample.
\end{theorem}
This will be obtained as a consequence of a generalized canonical bundle formula (Theorem 14) to be proved in the third section.

\begin{remark} 
Thus in particular, if $(X,B)$ is a klt pair with $K_X+B$ effective and if $L \in \Pic X$ is nef such that $K_X+B+L$ nef and abundant, then $K_X+B+L$ is semiample. The non-vanishing conjecture (\cite[page 1] {LP}) predicts that assuming $K_X+B$ pseudoeffective would suffice for this to hold.
\end{remark}

\section{Preliminaries} We will always work with $\mathbb{Q}$-divisors.

\begin{definition}
[Abundant divisors] Let $X$ be a normal projective variety and $D$ be a  nef $\mathbb{Q}$-divisor with Iitaka dimension $\kappa(X,D)$. The \emph{numerical dimension} of $D$ is defined as $\nu(X,D)= max\{e|D^{e} \not\equiv 0\}$. In general, $\kappa(X,D) \leq \nu(X,D)$. $D$ is \emph{abundant} if $\kappa(X,D) = \nu(X,D)$. Examples include big nef divisors, semiample divisors and their pullbacks. 
\end{definition}

\begin{definition}
[b-divisors (see \cite{Co})] Let $X$ be a normal variety. A \emph{b-divisor} $\mathbf{D}$ on $X$ is a collection of Weil divisors $\mathbf{D}_Y$ on higher birational models $ Y \rightarrow X$ of $X$ that are compatible under pushforward. $\mathbf{D}_Y$ is called the \emph{trace} of $\mathbf{D}$ on $Y$. The \emph{Cartier closure} of a $\mathbb{Q}$-Cartier $\mathbb{Q}$-divisor $D$ on $X$ is the b-$\mathbb{Q}$ divisor $\overline{\mathbf{D}}$ with trace $\overline{\mathbf{D}}_Y :=\rho ^{*}(D)$ on a birational model $\rho: Y \rightarrow X$. A b-$\mathbb{Q}$-divisor $\mathbf{D}$ on a normal variety $X$ is \emph{b-$\mathbb{Q}$-Cartier} if it is the Cartier closure of a $\mathbb{Q}$-Cartier $\mathbb{Q}$-divisor on some birational model of $X$; in this situation, we say that $\mathbf{D}$ \emph{descends} to $Y$. $\mathbf{D}$ is called \emph{b-nef} (resp. \emph{b-semiample}) if it is descends to a nef (resp. semiample) $\mathbb{Q}$-divisor on some birational model over $X$. Let $(X,B)$ be a sub-pair consisting of a normal variety $X$ and a $\mathbb{Q}$-divisor $B$ on $X$ such that $K_X+B$ is $\mathbb{Q}$-Cartier. The \emph{discrepancy b-divisor} $\mathbf{A} = \mathbf{A}(X,B)$ is defined by $K_Y = f^*(K_X+B)+\mathbf{A}_Y$ for all models $Y \xrightarrow{f} X$. $(X,B)$ is called \emph{sub-klt} (resp \emph{sub-lc}) if all coefficients of $\mathbf{A}_Y(X,B)$ are $>-1$ (resp. $\geq -1$) for some (and hence, by \cite[Lemma 2.30]{KM}, for all) birational models $ f: Y \rightarrow X$. Thus in particular, if $X$ is smooth, all coefficents of $B$ are less than $1$ and $B$ has normal crossing support, then $(X,B)$ is sub-klt. In case $B \geq 0$, we drop the prefix sub. 
\end{definition}

\begin{definition}
[Generalized pairs and their singularities(see \cite{Bi})] A \emph{generalized sub-pair} $(X,B+\mathbf{M})$ consists of a normal projective variety $X$, a $\mathbb{Q}$-divisor $B$ and a $\mathbb{Q}$-b-divisor $\mathbf{M}$ on $X$ such that:
\begin{enumerate}
\item $K_X+B+\mathbf{M}_X$ is $\mathbb{Q}$-Cartier.
\item $\mathbf{M}$ is b-nef.
\end{enumerate}
When $B \geq 0$, we drop the prefix sub.\\

Let $(X,B)$ be a generalized sub-pair and $Y \xrightarrow{\mu} X$ a higher birational model of $X$. Let $B_Y$ be defined by $K_Y+B_Y+\mathbf{M}_Y = \mu^*(K_X+B+\mathbf{M}_X)$. We say that $(X,B+\mathbf{M})$ is \emph{sub-generalized klt} (resp. \emph{sub-generalized lc}) if every coefficient of $B_Y$ is less than $1$ (resp. less than or equal to $1$). We will call them sub-gklt and sub-glc in short. The \emph{generalized discrepancy b-divisor} $\mathbf{A}(X,B+\mathbf{M})$ is then defined by $\mathbf{A}_Y(X,B+\mathbf{M})=-B_Y$. Note that $\mathbf{M}_Y$ does not contribute to the singularities of $(Y,B_Y+\mathbf{M})$ once it descends to a nef divisor on $Y$, i.e. $\mathbf{A}(Y,B_Y+\mathbf{M}_Y)=\mathbf{A}(Y,B_Y)$ and thus $(Y,B_Y+\mathbf{M}_Y)$ is gklt (resp. glc) iff $(Y,B_Y)$ is klt (resp. lc).
\end{definition}
\begin{definition}
[Generalized klt-trivial fibrations] Let $(X,B)$ be a pair consisting of a normal variety $X$ and a $\mathbb{Q}$-divisor $B$ on $X$. A morphism $f:(X,B) \rightarrow Y$, $Y$ normal is \emph{ generalized klt-trivial} if
\begin{enumerate}
\item $f$ is surjective with connected fibers,
\item $(X,B)$ is sub-klt over the generic point of $Y$,
\item there exists a $\mathbb{Q}$-Cartier $\mathbb{Q}$-divisor $D$ on $Y$ and a nef $\mathbb{Q}$-divisor $L$ on $X$ such that $K_X+B+L \sim_{\mathbb{Q}}f^*D$
\item there exists a log-resolution $\pi^{'}: (X^{'}, B_{X^{'}}) \rightarrow (X,B)$ such that if $f^{'}: (X^{'},B_{X^{'}}) \rightarrow Y$ is the induced morphism, then rank $f^{'}_* \mathcal{O}_{X^{'}}(\left\lceil{-B_{X^{'}}}\rceil\right) =1$.
\end{enumerate}
\end{definition}
The following fact seems to be well-known to the experts. I will include a proof for the convenience of readers and for want of a reference. 

\begin{lemma}Let $f:(X,B) \rightarrow Y$ be a generalized klt-trivial fibration, $ \sigma : Y^{'} \rightarrow Y$ a generically finite morphism from a normal variety $Y^{'}$. Let $X^{'}$ denote a desingularization of the main component of $ X \times _Y Y^{'}$ and let $f^{'}:(X^{'},B_{X^{'}}) \rightarrow Y^{'}$ be the induced morphism. Then $f^{'}$ is generalized klt-trivial.
\begin{proof} Only item $(4)$ in the above definition needs to be verified. Since $(X^{'},B_{X^{'}})$ is sub-klt, the rank is surely $ \geq 1$. We only need to show it is $\leq 1$ and may assume that $X$ itself is smooth. Let $\pi: (X^{'}, B_{X^{'}}) \rightarrow (X,B)$ be the induced morphism. 

\begin{enumerate}
\item Suppose first that $ \sigma$ is birational. Then there exists an effective, integral, $\pi$-exceptional divisor $E^{'}$ on $X^{'}$ such that $K_{X^{'}}=\pi^*(K_X)+E^{'}$. The condition $K_{X^{'}}+B_{X^{'}}= \pi^*(K_X+B)$ then gives $B_{X^{'}}+E^{'}= \pi^*B_X$. If we can show that $\sigma \circ f^{'}: X^{'} \rightarrow Y^{'}$ satisfies rank $(\sigma \circ f^{'})_* \mathcal{O}_{X^{'}}(\left \lceil {-B_{X^{'}}} \right\rceil)=1$, then rank $f^{'}_* \mathcal{O}_{X^{'}}(\left\lceil{-B_{X^{'}}}\right\rceil)=1$ also since $\sigma \circ f^{'}$ and $f^{'}$ have the same general fiber. Thus we may assume $\sigma =id$ and thus $f \circ \pi = f^{'}$. Now,  

\begin{center} 
$f^{'}_*\mathcal{O}_{X^{'}}(\left\lceil{-B_{X^{'}}}\right\rceil)= f^{'}_*\mathcal{O}_{X^{'}}( \left\lceil{-\pi^*B_X+E^{'}} \right\rceil)$ 
\end{center}

\begin{center}
$=f^{'}_* \mathcal{O}_{X^{'}}( \left\lceil{-\pi^*B_X}\right \rceil +E^{'})$ (since $E^{'}$ is integral)
\end{center}

\begin{center}
$\subset f^{'}_* \mathcal{O}_{X^{'}}(\pi^* \left\lceil{-B_{X}}\right\rceil +E^{'})$ 
\end{center}
\begin{center}
$=f_* \pi_* \mathcal{O}_{X^{'}}(\pi^*(\left\lceil{-B_{X}}\right \rceil) +E^{'}) = f_*\mathcal{O}_X(\left \lceil {-B_X}\right\rceil) \otimes \pi_*\mathcal{O}_{X^{'}}(E^{'}))$
\end{center}
\begin{center}
$ =f_*\mathcal{O}_X(\left\lceil{-B_X}\right \rceil)$ (by Hartog's theorem since $E^{'}$ is $\pi$-exceptional).
\end{center}
Thus rank $f^{'}_*\mathcal{O}_{X^{'}}(\left \lceil {-B_{X^{'}}} \right \rceil) \leq 1$.

\item Now suppose that $\sigma$ is finite. In this case also, there exists an integral divisor $R^{'} \geq 0$ such that $K_{X^{'}} = \pi^*(K_X)+R^{'}$ (Supp ($R^{'}$) is the ramification locus of $\pi$). This gives $B_{X^{'}}+R^{'} =\pi^*B_X$. To show that rank $f^{'}_*\mathcal{O}_{X^{'}}(\left \lceil {-B_{X^{'}}} \right \rceil) =1$, it suffices to show that $R^{'}$ is not $f^{'}$-horizontal. Now \cite[Chapter 3, Proposition 10.1 (b)]{Hart} implies that $f^{-1}(Sm(\sigma)) \subset Sm (\pi)$ ($Sm$ denotes the smooth locus of the respective morphism) or in other words, $f(Branch(\pi))=f(\pi(R^{'})) \subset Branch (\sigma)$. Thus $R^{'}$ can not be $f^{'}$-horizontal.

\end{enumerate}

\end{proof}

\end{lemma}

\section{Nef and abundant divisors} 

We start off by proving two easy results. They provide positive evidence for Question 1. Recall that an \emph{almost strictly nef} (ASN) divisor means a birational pullback of a strictly nef Cartier divisor. See \cite {PC}, \cite{Zh}, \cite{HL} for some recent developments around ASN divisors.

\begin{lemma}
Let $X$ be a normal projective variety, $M \in \Pic X$ be almost strictly nef (ASN) and abundant. Then $M$ is big.
\begin{proof}
By \cite[Proposition 2.1]{Ka}, there exist morphisms $X \xleftarrow {\mu} X^{'} \xrightarrow{f} Y$, where $\mu$ is birational, $f$ a contraction and $X^{'}$ and $Y$ are smooth and projective such that 
\begin{equation} \label{5}
M^{'} := \mu^* M = f^*(M_Y)
\end{equation}
for some $M_Y \in \Pic Y$ nef and big. If $\dim Y =\dim X$, then $f$ is generically finite and $M$ is already big. So assume that $\dim Y < \dim X$. By definition of ASN, there exists $ X \xrightarrow{\pi} X_0$ birational and $M_0 \in \Pic X_0$ strictly nef such that $ \pi^*(M_0) = M$. Let $E \subset X$ be the exceptional locus of $ \pi$. Then for all curves $C \not \subset E$, $(M \cdot C)= ( \pi^*(M_0) \cdot C) >0$. If $F \subset X^{'}$ is a general fiber of $f$, then $F \not \subset E$. Then there is a curve $C \subset F$ such that $ C \not \subset E$ and thus $(M \cdot C) >0$. This is not possible by (\ref{5}). Therefore, $\dim Y = \dim X$ and $M$ is big.
\end{proof}
\end{lemma}

\begin{corollary}Let $(X,B)$ be a klt pair, $ L \in \Pic X$ strictly nef such that $K_X+B+t_0L$ is abundant for some $t_0 > 2 \cdot \dim X$. Then $K_X+B+tL$ is ample  for all $ t > 2 \cdot  \dim X$.

\begin{proof} $K_X+B+t_0L$ is strictly nef by the cone theorem and abundant, hence big by above lemma. Then $2(K_X+B+t_0L)-(K_X+B)$ is nef and big and thus $K_X+B+t_0L$ is semiample by the basepoint free theorem. Now it is easy to see that a strictly nef and semiample divisor is ample, thus $K_X+B+t_0L$ is ample. The conclusion now follows from \cite[Lemma 1.3]{Ser} (or rather its log version which has the same proof).
\end{proof}
\end{corollary}

The following generalization of the basepoint-free theorem will be indispensible for us. It can be proved by a minor modification of the usual arguments (\cite[page 78]{KM}). See \cite [Theorem 2.1]{Am2} for the proof of a much more general statement.
\begin{theorem} \label{Amb}Let $(X,B)$ be a sub-klt pair, $H$ a nef $\mathbb{Q}$-Cartier divisor on $X$. Assume that 
\begin{enumerate}
\item $rH-(K_X+B)$ is nef and big for some $r \in \mathbb{N}$.
\item $H^0(\mathcal{O}_X(\left\lceil{\mathbf{A}(X,B)} \right\rceil +j \overline{\mathbf{H}})) \subset H^0(\mathcal{O}_X(jH))$ for all $ j \in \mathbb{N}$.
\end{enumerate}
Then $H$ is semiample.
\end{theorem}
 
\begin{lemma} \label{Fuj}
Let $(X,B+\mathbf{M})$ be a sub-generalized klt pair, $H$ nef $\mathbb{Q}$-Cartier divisor on $X$ such that $H-(K_X+B+\mathbf{M}_X)$ is nef and abundant. Suppose also that $\kappa(aH-(K_X+B+\mathbf{M}_X)) \geq 0$ and 
$ \nu(aH-(K_X+B+\mathbf{M}_X)) =\nu(H-(K_X+B+\mathbf{M}_X))$
for some $a \in \mathbb{Q}_{>1}$. Let 
\begin{center} $(X, B+\mathbf{M}_X) \xleftarrow{\mu}(Y, B_Y+\mathbf{M}_Y) \xrightarrow{f} Z$ \end{center}
be the standard diagram induced by the Iitaka fibration of $H-(K_X+B+\mathbf{M}_X)$ (see \cite[Proposition 2.1]{Ka}). Then we have the following:

\begin{enumerate}
\item Suppose that the moduli b-divisor $\mathbf{M}_Z$ induced on $Z$ by the generalized sub-pair $(Y,B_Y+\mathbf{M}_Y)$ is b-nef. Then $H$ is semiample.
\item Suppose $(X,B+\mathbf{M})$ is generalized klt. Then rank $f_*\mathcal{O}_Y(\left\lceil{-B_Y}\rceil\right)=1$.
\end{enumerate}

\begin{proof}
We argue along the lines of \cite{Fu}.\vspace{0.3 cm}

Recall first that \cite[Proposition 2.1,2.3]{Ka} says that there exist smooth projective varieties $Y$ and $Z$ and morphisms 
\begin{center}
$(X,B+\mathbf{M}_X) \xleftarrow{\mu} (Y,B_Y+\mathbf{M}_Y) \xrightarrow{f} Z$
\end{center}
where $\mu$ is birational, $f$ is a surjective morphism with connected fibers birational to the Iitaka fibration of $H-(K_X+B+\mathbf{M}_X)$, $\mu^*(K_X+B+\mathbf{M}_X)=K_Y+B_Y+\mathbf{M}_Y$, where $B_Y$ is defined by $\mu_*(\mathbf{M}_Y)=\mathbf{M}_X$. We may assume that $\mathbf{M}_Y$ descends to a nef Cartier divisor on $Y$ (possibly after replacing $Y$ by a higher model) and there exists $D_Z \in \Pic (Z)$ nef and big such that, setting $H_Y=\mu^*(H)$, we have
\begin{center}
 $\mu^*(H-(K_X+B+\mathbf{M}_X))=H_Y-(K_Y+B_Y+\mathbf{M}_Y) \sim_{\mathbb{Q}} f^*(D_Z)$.
\end{center}
Also, there exists $H_Z \in \Pic (Z)$ nef such that $H_Y \sim_{\mathbb{Q}} f^*(H_Z)$. Note that $(Y,B_Y)$ is sub-klt and 
\begin{center}
$\mathbf{A}_Y(Y,B_Y+\mathbf{M}_Y)=-B_Y$
\end{center}

As noted above, there exists a nef and big divisor $D_Z$ and a nef divisor $H_Z$ on $Z$ such that $K_Y+B_Y+\mathbf{M}_Y= f^*(H_Z-D_Z)$. Now by assumption, we can write $H_Z-D_Z \sim_{\mathbb{Q}} K_Z+B_Z+\mathbf{M}_Z$ such that \\

\begin{enumerate}
\item $\mathbf{M}_Z$ is $\mathbb{Q}$-Cartier and b-nef.
\item $(Z,B_Z)$ is sub-klt. $B_Z$ is defined as follows. Let $D \subset Z$ be any prime divisor. Let 
\begin{center}
$a_D:=sup\{t|(Y, B_Y+tf^*(D))$ is sub-lc over the generic point of $D\}$
\end{center}
and define $B_Z= \Sigma (1-a_D)D$. We note that replacing $Y$ and $Z$ by higher models, we may assume that $(Y,B_Y)\xrightarrow{f} (Z,B_Z)$ satisfies standard SNC assumptions (see \cite[Definition 4.4]{Fi}). Then the definition of $a_D$ requires that all coefficients of $B_Y+tf^*(D)$ are $\leq 1$. Now, $(Y,B_Y)$ being sub-klt implies that all coefficients of $B_Y$ are $<1$ and thus that $a_D >0$ for all $D$. $(Z,B_Z)$ is thus sub-klt.
\end{enumerate}
Now let $(Z^{'}, B_{Z^{'}}) \rightarrow (Z,B_Z)$ be a higher model such that $\mathbf{M}_{Z^{'}}$ is nef. Then $(Z^{'}, B_{Z^{'}})$ is of course also sub-klt. Let $Y^{'}$ be the normalization of the main component of $Y \times _{Z}Z^{'} \rightarrow Z^{'}$. Replacing $Y$ and $Z$ by $Y^{'}$ and $Z^{'}$, we may assume that $M_Z$ is nef on $Z$. We need:
\begin{lemma} (See \cite[Lemma 9.2.2]{Co}) Let $f:(Y,B_Y) \rightarrow Z$ be a fibration such that $(Y,B_Y)$ is sub-klt over the generic point of $Z$. Let $\mathbf{B} $ be the induced discriminant $\mathbb{Q}$-b-divisor of $Z$. Then $\mathcal{O}_Z(\left\lceil{-\mathbf{B}}\right\rceil)\subset f_* \mathcal{O}_Y(\left\lceil{\mathbf{A}(Y,B_Y)} \right\rceil)$.
\end{lemma}
As in the proof of \cite[Proposition 9.2.3]{Co}, this shows that $\mathcal{O}_Z(\left\lceil\mathbf{{-B+D}} \right\rceil)$\\
$ \subset f_*\mathcal{O}_Y(\left\lceil{\mathbf{A}(Y,B_Y)+f^*\mathbf{D}}\right\rceil)$ holds for every $\mathbb{Q}$-b-Cartier $\mathbb{Q}$-b-divisor $\mathbf{D}$ on $Y$: we can replace $Z$ by a higher model $Z^{'}$ on which $\mathbf{D}$ is $\mathbb{Q}$-Cartier and replace $Y$ by the normalization of the main component of $Y \times_Z Z^{'}$ as before. Then $(Y, B_Y-f^*D)$ is sub-klt over the generic point of $Z$ with discriminant $\mathbb{Q}$-b-divisor $\mathbf{B-D}$ and \begin{center}
$\mathbf{A}(Y,B_Y-f^*D)=\mathbf{A}(Y,B_Y)+f^*(D)$. 
\end{center}
Now the above lemma applied to $f:(Y,B_Y-f^*D) \rightarrow Z$ gives 
\begin{center}
$\mathcal{O}_Z(\left\lceil{\mathbf{-B+D}}\right\rceil) \subset f_*\mathcal{O}_Y(\left\lceil{\mathbf{A}(Y,B_Y)+f^*\mathbf{D}}\right\rceil)$
\end{center}
as claimed. Thus 
\begin{center}
$\mathcal{O}_Z(\left\lceil{\mathbf{A}(Z,B_Z)}\right\rceil +j \overline{\mathbf{H}_Z}) \subset f_*\mathcal{O}_Y(\left\lceil {\mathbf{A}(Y,B_Y)}\right\rceil + j \overline{\mathbf{H}_Y})$
\end{center}
for all $ j \in \mathbb{Z}$. Taking $H^0$ gives 
\begin{center}
$H^0(Z, \mathcal{O}_Z(\left\lceil{\mathbf{A}(Z,B_Z)}\right\rceil +j \overline{\mathbf{H}_Z}) )\subset H^0(Y, \mathcal{O}_Y(\left\lceil {\mathbf{A}(Y,B_Y)}\right\rceil + j \overline{\mathbf{H}_Y})) $
\end{center}
\begin{center}
$=H^0(X, \mu_*(\mathcal{O}_Y(\left\lceil {\mathbf{A}(Y,B_Y)}\right\rceil + j \overline{\mathbf{H}_Y})) = H^0(X, \left \lceil{A(X,B)}\right\rceil +jH)$
\end{center}
\begin{center}
$ =H^0(X, \mathcal{O}_X(jH)) =H^0(Y, \mathcal{O}_Y(jH_Y))$
\end{center}
\begin{center}
$=H^0(Y,\mathcal{O}_Y(jf^*(H_Z)))=H^0(Z, \mathcal{O}_Z(jH_Z))$.
\end{center}
Thus $H_Z$ is semiample by Theorem \ref{Amb}. This proves ($1$). \\

Now we prove ($2$). Clearly rank $f_* \mathcal{O}_Y(\left \lceil {-B_Y}\right \rceil) \geq 1$ since $(Y, B_Y)$ is sub-klt. \\

Take $ A \in \Pic Z$ sufficiently ample such that $f_* \mathcal{O}_Y(\left \lceil {-B_Y} \right \rceil ) \otimes A$ is globally generated. Since $D_Z$ is big, 

\begin{center}
$f_* \mathcal{O}_Y(\left \lceil {-B_Y} \right \rceil ) \otimes A \subset f_* \mathcal{O}_Y(\left \lceil {-B_Y} \right \rceil ) \otimes \mathcal{O}_Z(mD_Z)$
\end{center}

if $m \gg 0$. Now note that since $B_Y= \mu_*^{-1}(B)+F$ where $F$ is $\mu$-exceptional, $B \geq 0$ and 
\begin{center}
$ 0 \leq \left \lceil {-B_Y} \right \rceil = \left \lceil {\mu_*^{-1}(-B)} \right \rceil + \left \lceil{-F} \right \rceil $,
\end{center}
 we can conclude that $\left\lceil{\mu_*^{-1}(-B)}\right \rceil =0$ and hence $ \left \lceil {-B_Y} \right \rceil $ is an effective $\mu $-exceptional divisor.
 Now 
\begin{center}
$ H^0(Z, f_* \mathcal{O}_Y(\left \lceil {-B_Y} \right \rceil ) \otimes A) \subset H^0(Z, f_* \mathcal{O}_Y(\left \lceil {-B_Y} \right \rceil ) \otimes \mathcal{O}_Z(mD_Z)) = H^0(Y, \mathcal{O}_Y(\left \lceil {-B_Y} \right \rceil \otimes  f^*(mD_Z))$.
\end{center}

We have 
\begin{center}
$ H^0(Y, \mathcal{O}_Y(\left \lceil {-B_Y} \right \rceil \otimes \mathcal{O}_Y( f^*(mD_Z))) = H^0(\mathcal{O}_Y(E) \otimes f^*(mD_Z)) $
\end{center} where $E \geq 0$ is $ \mu $-exceptional. 
 This equals 
\begin{center}
$H^0(Y, \mathcal{O}_Y(E) \otimes \mu^*(m(H-(K_X+B+\mathbf{M}_X)))) = H^0(X, m(H-(K_X+B+\mathbf{M}_X)))= H^0(Z, mD_Z)$.

\end{center}
This shows that rank $f_* \mathcal{O}_Y(\left \lceil{-B_Y} \right \rceil) =1$.

\end{proof}
\end{lemma}

\begin{corollary}Let $(X,B+\mathbf{M})$ be a generalized klt pair, $H$ a nef $\mathbb{Q}$-Cartier divisor on $X$ such that $H-(K_X+B+\mathbf{M}_X)$ is nef and abundant. Assume further that $\mathbf{M}$ is b-semiample, $\kappa(aH-(K_X+B+\mathbf{M}_X)) \geq 0$ and $ \nu(aH-(K_X+B+\mathbf{M}_X)) =\nu(H-(K_X+B+\mathbf{M}_X))$ for some $a \in \mathbb{Q}_{>1}$. Then $H$ is semiample.

\begin{proof}

As in the proof of the previous Lemma, there exist morphisms

\begin{center} 
$(X,B+\mathbf{M}_X) \xleftarrow{\mu} (Y, B_Y+ \mathbf{M}_Y) \xrightarrow{f} Z$
\end{center}
where $\mathbf{M}_Y$ is semiample, $Y$ and $Z$ are both smooth, there exists $D_Z \in \Pic Z$ and nef and big, $H_Z \in \Pic Z$  such that $H_Y \sim_{\mathbb{Q}} f^*(H_Z)$,
\begin{center}
$\mu^*(H-(K_X+B+\mathbf{M}_X))=H_Y-(K_Y+B_Y+\mathbf{M}_Y) \sim_{\mathbb{Q}} f^*(D_Z)$
\end{center}
and
\begin{center}
 $\mu^*(K_X+B_X+\mathbf{M}_X)= K_Y+B_Y+\mathbf{M}_Y \sim _{\mathbb{Q}} f^*(H_Z-D_Z)$
\end{center} 
such that Supp($B_Y$) is SNC. Note: $\left\lceil{-B_Y} \right \rceil \geq 0$. Since $\mathbf{M}_Y$ is semiample, for $m$ sufficiently large and divisible, there exists $H=D_1+ \cdots +D_k \in |m\mathbf{M}_Y|$ such that $D_i$ is smooth for all $i$, $D_i \cap D_j = \varnothing $ for all $i \ne j$ and Supp $H$ $\cup $ Supp $B_Y$ is SNC. By taking $m \gg 0$, $(1/m) H \in |\mathbf{M}_Y| _{\mathbb{Q}}$ can be made arbitarily small, thus $(Y,B_Y+(1/m)H)$ is sub-klt. Now consider the induced fibration $(Y, B_Y+(1/m)H) \xrightarrow{f} Z$. Then 
\begin{center}
$f_* \mathcal{O}_Y(\left \lceil {-(B_Y+(1/m) H)} \right \rceil) = f_* \mathcal{O}_Y(\left \lceil {-B_Y}\right \rceil)$ 
\end{center}
 since $H$ can be chosen to have no components in common with $B_Y$. By Lemma \ref{Fuj}($2$), we have rank $f_*\mathcal{O}_Y(\left\lceil{-B_Y} \right \rceil ) =1$. Thus $(Y, B_Y+1/mH) \xrightarrow{f} Z$ is a klt-trivial fibration. If $B_Z$ and $\mathbf{M}_Z$ denote the induced discriminant and the moduli divisors on $Z$, then by \cite[Theorem 2.7]{Am}, $H_Z-D_Z= K_Z+B_Z+\mathbf{M}_Z$ where $\mathbf{M}_Z$ is b-nef. Thus we are done by the $\mathbf{M}=0$ case of the Lemma \ref{Fuj}($1$) (which is essentially worked out in \cite{Fu}).

\end{proof}

\end{corollary}

\begin{remark} The above corollary can also be proved by combining \cite [Theorem 4.12]{Fi} with Lemma \ref{Fuj}($1$).
\end{remark}

The following is a generalized canonical bundle formula which is used for proving theorem 4. It may also be of independent interest.

\begin{theorem} Let $(Y, B)$ be a sub klt pair, $L \in \Pic X$ nef and $f:Y \rightarrow Z$ be a contraction such that $K_Y+B+L \sim_{\mathbb{Q},f} 0$ and rank $f_*\mathcal{O}_Y(\left\lceil{\mathbf{A}(Y, B)}\right\rceil) =1$. Suppose $K_Y+B \geq_{\mathbb{Q}}0$. Then the moduli b-divisor $\mathbf{M}_Z$ induced on $Z$ is b-nef.
\end{theorem}
\begin{proof} We write
\begin{equation} \label{1}
K_Y+B_Y+L_Y \sim_{\mathbb{Q}} f^*(D_Z) = f^*(K_Z+B_Z+\mathbf{M}_Z).
\end{equation} as before. Let $F \subset Y$ be a general fiber of $f$. Then $(K_Y+B_Y)|_{F} \geq 0$. Suppose $(K_Y+B_Y)|_{F} >_{\mathbb{Q}}0$. Then there exists a curve $C \subset F$ such that $((K_Y+B_Y)|_F \cdot C) >0$ and thus $((K_Y+B_Y+L_Y)|_F \cdot C) >0$. This contradicts (\ref{1}). Thus $(K_Y+B_Y)|_F \sim_{\mathbb{Q}} 0$, $L|_F \sim_{\mathbb{Q}} 0$, where $F$ is a general fiber of $f$. Let $ d \in \mathbb{N}$ such that $\mathcal{O}_F(dL_Y|_F) \cong \mathcal{O}_F$. Now by \cite[Lemma 3.1]{LP}, after replacing $Y$ and $Z$ with higher models, we can assume that $L_Y$ is numerically trivial on every fiber of $f$. By performing a flattening base change, we can replace $Y$ and $Z$ again by higher models and assume that $f$ is flat. Note that $\mathcal{O}_F(dL_Y|_F) \cong \mathcal{O}_F$ is still true for a general fiber $F$. Then by the semicontinuity theorem, 
\begin{center}
$ z \rightarrow h^0(\mathcal{O}(d L_Y)|_{Y_{z}})$
\end{center}
 is an upper semicontinuous function on $Z$. Thus 
\begin{center}
$\{z \in Z| h^0(\mathcal{O}(dL_Y|_{Y_z})) \geq 1 \}$
\end{center}
is closed in $Z$. Since $\mathcal{O}_Y(d L_Y)$ is trivial on a general fiber and numerically trivial on all fibers, this shows that $\mathcal{O}_Y(dL_Y|_{Y_{z}}) \cong \mathcal{O}_{Y_{z}}$ for all $z \in Z$. Now there exists a finite morphism $ \theta: Z^{'} \rightarrow Z$, from a smooth projective variety $Z^{'}$ such that letting $Y^{'}$ denote a desingularization of the main component of $Y \times _{Z} Z^{'}$, the induced morphism $f^{'}: Y^{'} \rightarrow Z^{'}$ is semistable in codimension 1. Since $\mathcal{O}_{Y^{'}}(dL_{Y^{'}})$ is trivial on all fibers of $f^{'}$ and $f^{'}$ is flat with reduced connected fibers over a large open set, there exists $P_{Z^{'}} \in \Pic (Z^{'})$ nef such that $f^{'*}(P_{Z^{'}}) = \mathcal{O}_{Y^{'}}(dL_{Y^{'}})$ (same arguments as \cite[Chapter 3, Exercise 12.4]{Hart}). Let $ \alpha: Z^{''} \rightarrow Z^{'}$ be a Bloch-Gieseker cover with $Z^{''}$ smooth such that $\alpha^*(P_{Z^{'}})= L_{Z^{''}}^{\otimes d}$ for some $L_{Z^{''}} \in \Pic (Z^{''})$ (see \cite [Theorem 4.1.10]{Laz}). Then $L_{Y^{''}}= f^{''*}(L_{Z^{''}})$ where $f^{''}: Y^{''} \rightarrow Z^{''}$ is obtained by base changing $f$. \\

We now argue along the lines of \cite[Lemma 5.2(5)]{Am} to show that $\mathbf{M}_Z$ is b-nef. Let $b \in \mathbb{N}$ smallest such that $b(K_{Y^{''}}+B_{Y^{''}}+L_{Y^{''}}) \sim_{f^{''}} 0$. There exists $ \phi \in K(Y^{''})^*$ such that $K_{Y^{''}}+B_{Y^{''}}+L_{Y^{''}}+ 1/b (\phi) = f^{''*}(K_{Z^{''}}+B_{Z^{''}}+\mathbf{M}_{Z^{''}})$. Let $V^{''}$ be a desingularization of the normalization of $Y^{''}$ in $K(Y^{''})(\phi^{1/b})$. Note that $V^{''}$ is irreducible by the minimality of $b$. Let 

\begin{center} 
$\pi^{''}: (V^{''},B_{V^{''}}) \xrightarrow{\nu^{''}} (Y^{''}, B_{Y^{''}}) \xrightarrow{f^{''}} Z^{''}$ 
\end{center}
be the induced set-up.\\

Let $\Sigma_{f ^{''}}\subset  Z^{''}$ be the maximal closed subvariety such that $f^{''}$ is smooth over $Z^{''} \setminus  \Sigma_{f^{''}}$. Replacing $Z^{''}$ by a log-resolution, we can assume that $ \Sigma _{f^{''}}$ has SNC support. Also by \cite[Proposition 5.4]{Am}, there exists $\tau : \overline{Z^{'}} \rightarrow Z^{''}$ finite such that $\overline{Z^{'}}$ admits an SNC divisor supporting $\tau^{-1}(\Sigma_{f^{''}})$ and the locus where $\tau$ is not etale such that the induced morphism $(\overline{V^{'}}, B_{\overline{V^{'}}}) \xrightarrow {\overline{\pi^{'}}} \overline {Z^{'}}$ is semistable in codimension 1. Then $\overline{\pi^{'}}_*(\omega_{\overline {V^{'}}/\overline{Z^{'}}})$ is nef (\cite[Theorem 4.4]{Am}). If $\mathbf{M}_{\overline{Z^{'}}}$  is the induced moduli b-divisor on $\overline{Z^{'}}$, then $\tau^* \mathbf{M}_{Z^{''}}= \mathbf{M}_{\overline{Z^{'}}}$ (\cite [Proposition 5.5] {Am}) and $( \theta \circ \alpha)^*(\mathbf{M}_{Z^{''}})= \mathbf{M}_{Z}$ (\cite[Lemma 5.1]{Am}). So it suffices to show that $\mathbf{M}_{\overline{Z^{'}}}$ is nef. The category of generalized klt-trivial fibrations is  closed under generically finite base changes (Lemma 5). Thus $\overline{f^{'}}: (\overline{Y^{'}}, B_{\overline {Y^{'}}}) \rightarrow  \overline{Z^{'}}$ is again generalized klt-trivial.\\

 To simplify notation, we thus replace $\overline{V^{'}}, \overline{Y^{'}}, \overline{Z^{'}}, \overline {\nu^{'}}, \overline{\pi^{'}}, \overline{f^{'}}$ by $V, Y,Z, \nu, \pi, f$ respectively. \\

Then $ \pi_*(\omega_{V/Z})$
\begin{center}
$= f_*(\bigoplus _{i=0}^{b-1} \mathcal{O}_Y(K_Y+\left\lceil{i/b(\phi)}\right\rceil) \otimes
 f^*\omega _Z ^{-1}) \otimes f_*(\bigoplus _{i=0}^{b-1} \mathcal{O}_Y(\left\lfloor {-i/b(\phi)} \right \rfloor ))$
\end{center} (\cite[Lemma 2.3]{Vi}).
 We have
\begin{center}
 $f_*(\bigoplus _{i=0}^{b-1} \mathcal{O}_Y(\left\lfloor {-i/b(\phi)} \right \rfloor ))= \mathcal{O}_Z$ 
\end{center}
and 
\begin{center}
$i/b(\phi) = i(-K_{Y/Z}-B_Y-L_Y+f^*(B_Z+\mathbf{M}_Z))$.
\end{center}
from which it follows that 

$\pi_*(\omega_{V/Z})$
\begin{center}
$= \bigoplus _{i=0}^{b-1} f_*\mathcal{O}_Y(\left\lceil{(1-i)K_{Y/Z}-iB_Y-iL_Y+if^*(B_Z+\mathbf{M}_Z)}\right\rceil )$
 \end{center}
where the $i$-th summand on the right hand side is the eigensheaf corresponding to the eigenvalue $\zeta ^i$, where $\zeta$ is a primitive $b$-th root of unity, the action being that of $\mathbb{Z}/(b)$ on $K(V)$ and $\psi$ is a rational function on $V$ such that $\psi^b= \phi$ (see \cite [page 23-24]{EV} for more details). By \cite[Theorem 4.4]{Am}, $\pi_*(\omega_{V/Z})$ is a nef vector bundle and hence so are the direct summands. Then
\begin{equation} \label{2}
K_V+B_V+(\psi)= \pi^*(K_Z+B_Z+\mathbf{M}_Z-L_Z).
\end{equation}
Let 
\begin{center}
$\mathcal{L} = f_*\mathcal{O}_Y( \left\lceil {-B_Y-L_Y+f^*(B_Z+\mathbf{M}_Z)}\right\rceil) $
\end{center}

 be the $\zeta$-eigensheaf. Then $\mathcal{L} \in \Pic Z$ (it has rank $1$ since  rank $f_*\mathcal{O}_Y(\left\lceil{-B_Y} \right\rceil) =1$ and $L_Y$ is trivial on all fibers).
\begin{claim} There exists $ U \subset Z $ large open such that $(-B_V+ \pi^* B_Z)|_{\pi^{-1}(U)} \geq 0$.

\begin{proof} By (\ref{2}), we see that $(K_{V/Z})^h+B_V^h+(\psi)^h =0$ ($^h$ denotes the horizontal part) and hence that $B_V^h$ is a $\mathbb{Z}$-divisor. $(V, B_V)$ being klt, $\left\lceil{-B_V^h} \right\rceil= -B_V ^h \geq 0$. Let $B_V^v= \Sigma _j d_j P_j$ ($^v$ denotes the vertical part). We can ignore those $P_j$ such that codim $(\pi(P_j)) \geq 2$ and consider only those that map to some prime component of $B_Z$. Recall that $B_Z= \Sigma_l \delta_l Q_l$ where 
\begin{center}
$1-\delta _l = sup\{t| (Y,B_Y+t\pi^*(Q_l))$ is sub-lc over the generic point of $Q_l\}$.
\end{center}
which translates to $ \Sigma_j d_jP_j+ \pi^*(\Sigma_l(1-\delta_l)Q_l) \leq 1$. Fix $l$ and consider those $P_j$ which map to $Q_l$. Since $\pi$ is semistable in codimension $1$, this leads to $max _{P_j \rightarrow Q_l}(d_j)+1-\delta_l \leq 1$ with equality holding for some $j$. Thus $-d_j+\delta_l \geq 0$ for all $j$ such that $\pi(P_j)=Q_l$. This proves the claim.
\end{proof}
\end{claim}
Also note the above proof shows $(-B_V+ \pi^*B_Z)$ supports no fibers over codimension $1$ points of $Z$. Let $H \subset V$ be a general fiber of $\pi$. Then $(\psi|_H)+K_H=-B_H \geq 0$ (Note everything is integral here). Thus $\psi$ is a rational section of $\pi_*(\omega_{V/Z})$. Since $ \psi \rightarrow  \zeta \cdot \psi $ we conclude that $\psi$ is a rational section of $\mathcal{L}$. Let $\mathcal{K}(Z)$ be the (constant) sheaf of rational functions on $Z$.  Considering the injection $ \mathcal{L} \hookrightarrow \mathcal{K}(Z)$ that sends a rational section $e$ of $\mathcal{L}$ to $e/\psi $ shows that $\mathcal{L} \subset \mathcal{K}(Z) \cdot \psi $. \\

The above claim and (\ref{2}) show that there exists $U \subset Z$ large open such that $\mathcal{O}_Z(\mathbf{M}_Z-L_Z) \cdot \psi|_U \subset \pi_*(\omega_{V/Z})|_U$ and thus $\mathcal{O}_Z(\mathbf{M}_Z-L_Z) \cdot \psi|_U \subset \mathcal{L}|_U$. Conversely, let $a\cdot \psi $ be a  section of $\mathcal{L}$, $ a \in K(Z)$. Then it is a section of $\pi_*(\omega_{V/Z})$, so $(\pi^*(a) \cdot \psi)+K_{V/Z} \geq 0$. Since 
\begin{equation} \label{3}
  (\pi^*(a) \cdot  \psi)+K_{V/Z} = \pi^*((a)+\mathbf{M}_Z-L_Z)+ (-B_V+\pi^*B_Z)
\end{equation}
 we have $\pi^*((a)+\mathbf{M}_Z-L_Z)+ (-B_V+\pi^*B_Z) \geq 0$. Since $-B_V+\pi^*B_Z$ contains no fibers over codimension $1$ points of $Z$, this implies 
that $(a)+\mathbf{M}_Z-L_Z \geq 0$ and hence that $\mathcal{L} \subset \mathcal{O}_Z(\mathbf{M}_Z-L_Z) \cdot \psi$. We therefore conclude that 
\begin{center} 
$\mathcal{L}|_U= \mathcal{O}_Z(\mathbf{M}_Z-L_Z) \cdot \psi|_U$.
\end{center}
Since $U$ is large and $\mathcal{L}$ is a line bundle, $\mathcal{L} = \mathcal{O}_Z(\mathbf{M}_Z-L_Z) \cdot \psi$. Since $\mathcal{L}$ and $L_Z$ are nef, we conclude that $\mathbf{M}_Z$ is nef.

\end{proof}

\begin{corollary}
Let $(X, B)$ be a klt pair with $\kappa(K_X+B) \geq 0$, $H$ and $L$ nef $\mathbb{Q}$-Cartier divisors on $X$ such that $H-(K_X+B+L)$ is nef and abundant. Assume further that $\kappa(aH-(K_X+B+L)) \geq 0$ and $ \nu(aH-(K_X+B+L)) =\nu(H-(K_X+B+L))$ for some $a \in \mathbb{Q}_{>1}$. Then $H$ is semiample.

\begin{proof} Let $(X,B+L) \xleftarrow{\mu} (Y,B_Y+L_Y) \xrightarrow{f} Z$ be the standard diagram as before. Then rank $f_* \mathcal{O}_Y(\left \lceil{-B_Y} \right \rceil) =1$ by Lemma \ref{Fuj}($2$). Now the conclusion follows from Lemma \ref{Fuj}($1$) and the above theorem.

\end{proof}
\end{corollary}

\section{Acknowledgements}
I thank my advisor Patrick Brosnan for several enlightening discussions on the contents of this paper and Christopher Hacon, an email conversation with whom led to the start of this work. I would also like to express my gratitude to Florin Ambro, Caucher Birkar and Jason Starr for their valuable comments. I thank the referee for his comments which helped improve the presentation and readability of the paper. The author was partially supported by the Hauptman fellowship of the Mathematics department of University of Maryland during the preparation of this work.

 %%%%%%%%%%%%%%%%%%%%%%%%%%%%%%%%%%%%%%%%%%%%%%%%%%

\end{document}